\documentclass[12pt,regno]{amsart}
\usepackage{amsfonts,amsmath,amssymb,amsbsy,amstext,amsthm}
\usepackage{color,enumerate,float,fullpage,verbatim}
\usepackage[shortlabels]{enumitem}

\usepackage{url}
\usepackage[colorlinks=true,hyperindex, linkcolor=Mahogany, pagebackref=false, citecolor=NavyBlue]{hyperref}

\usepackage{eucal,bm,kpfonts,mathbbol}

\usepackage[dvipsnames]{xcolor}
\usepackage{verbatim}
\usepackage{color}
\usepackage[all]{xy}  
\usepackage{mathrsfs}

\theoremstyle{definition}
\newtheorem{theorem}{Theorem}[section]
\newtheorem{theoremx}{Theorem}

\numberwithin{equation}{section}

\newtheorem*{theorem*}{Theorem}

\newtheorem{lemma}[theorem]{Lemma}
\newtheorem{proposition}[theorem]{Proposition}

\newtheorem{algorithm}[theorem]{Algorithm}
\newtheorem*{claim*}{Claim}

\theoremstyle{definition}
\newtheorem{definition}[theorem]{Definition}

\newtheorem{example}[theorem]{Example}

\newtheorem{remarks}[theorem]{Remarks}

\newtheoremstyle{TheoremNum}
        {8pt}{8pt}              
        {\upshape}                      
        {}                              
        {\bfseries}                     
        {.}                             
        {.5em}                             
        {\thmname{#1}\thmnote{ \bfseries #3}}
  \theoremstyle{TheoremNum}

\newcommand{\RR}{\mathbb{R}}
\newcommand{\NN}{\mathbb{N}}
\newcommand{\ZZ}{\mathbb{Z}}

\newcommand{\rank}{\operatorname{rank}}

\newcommand{\syz}{\operatorname{syz}}

\renewcommand{\ker}{\operatorname{ker}}

\newcommand{\Ht}{\operatorname{ht}}

\newcommand{\ul}[1]{\underline{#1}}

\newcommand{\ov}[1]{\overline{#1}}

\renewcommand{\a}{\mathfrak{a}}

\renewcommand{\leq}{\leqslant}
\renewcommand{\geq}{\geqslant}

\newcommand{\kk}{\Bbbk}

\newcommand{\om}{\boldsymbol{\omega}}

\DeclareMathOperator{\Tor}{Tor}

\newcommand{\ale}[1]{{\color{red} \sf $\star$ Alessandro: [#1]}}

\newcommand{\f}{\mathfrak{f}}
\newcommand{\x}{\mathsf{x}}

\title[]{Bounds on the number of generators of prime ideals}
\author{Giulio Caviglia}
\address{Department of Mathematics, Purdue University, 150 N. University Street, West Lafayette, IN 47907-2067, USA}
\email{gcavigli@purdue.edu}
\author{Alessandro De Stefani}
\address{Dipartimento di Matematica, Universit{\`a} di Genova, Via Dodecaneso 35, 16146 Genova, Italy}
\email{destefani@dima.unige.it}
\thanks{The work of the first named author was partially supported by a grant from the Simons Foundation (41000748, G.C.)}

\subjclass[2010]{}

\begin{document}

\begin{abstract}
Let $S$ be a polynomial ring over any field $\kk$, and let $P \subseteq S$ be a non-degenerate homogeneous prime ideal of height $h$. When $\kk$ is algebraically closed, a classical result attributed to Castelnuovo establishes an upper bound on the number of linearly independent quadrics contained in $P$ which only depends on $h$. We significantly extend this result by proving that the number of minimal generators of $P$ in any degree $j$ can be bounded above by an explicit function that only depends on $j$ and $h$. In addition to providing a bound for generators in any degree $j$, not just for quadrics, our techniques allow us to drop the assumption that $\kk$ is algebraically closed. By means of standard techniques, we also obtain analogous upper bounds on higher graded Betti numbers of any radical ideal.
\end{abstract} 

\maketitle

\section{Introduction}

This article deals with the following fundamental question: {\it ``How many minimal generators of a given degree does a prime ideal in a polynomial ring have?''} 


A classical result attributed to Guido Castelnuovo states that, if $P$ is a non-degenerate (i.e., containing no linear form) homogeneous prime ideal of height $h$ in a polynomial ring $S$ over an algebraically closed field $\kk$, then the number of linearly independent quadrics contained in $P$ is at most ${h+1 \choose 2}$. What is relevant to observe is that the bound is independent of the number of variables of $S$, which for the purposes of this paper should be thought of as an unknown integer $n \gg 0$. 
A standard proof of Castelnuovo's theorem involves a repeated application of Bertini's theorem to guarantee that a general hyperplane section of the variety defined by $P$ is still non-degenerate, irreducible and reduced. In algebraic terms, going modulo a general linear form and saturating yields an ideal $\ov{P}$ in a polynomial ring in one less variable which is still prime, and does not contain any linear forms. For the latter, the assumption that $\kk$ is algebraically closed is crucial. 

Castelnuovo's estimate supports a more general philosophy which suggests that, under some reasonable geometric assumptions, prime ideals are expected to have better behavior than all other ideals. We point out, for instance, that this bound is completely false for general ideals, even radical; for instance the height-one ideal $(x_1x_j \mid j=2,\ldots,n) \subseteq \kk[x_1,\ldots,x_n]$ is radical, but it contains $n-1$ linearly independent quadrics. 

Castelnuovo's theorem resembles, in its spirit, two fundamental problems involving numerical invariants, both of which have recently been settled: the Eisenbud-Goto conjecture and Stillman's conjecture. The Eisenbud-Goto conjecture \cite{EG} states that, in the same geometric assumptions of Castelnuovo's theorem, the sum of the regularity of $S/P$ and the height of $P$ can be bounded above by the multiplicity of $S/P$. 
Recently, McCullough and Peeva \cite{McP} constructed a family of examples, based on two new constructions called step-by-step homogenization and Rees-like algebras and a series of examples due to Mayr-Meyer, which fail spectacularly in satisfying the bound predicted by the Eisenbud-Goto conjecture. In fact, they show that  the regularity of non-degenerate homogeneous prime ideals cannot be bounded by any polynomial function of the multiplicity. See also \cite{CCMcPV} for some further developments in this direction. While the conjecture is still open in several important cases, such as when $P$ defines a smooth projective variety, the striking result of McCullough and Peeva has put prime ideals under a completely new perspective.

Stillman's conjecture, on the other hand, states 
that if $I \subseteq S$ is a homogeneous ideal generated by $t$ forms of degrees $d_1,\ldots,d_t$, then the length of a graded minimal free resolution of $I$ over $S$ can be bounded above by a constant which only depends on $d_1,\ldots,d_t$ (see \cite{PS}). This conjecture has recently been proved in full generality by Ananyan and Hochster \cite{AH}, with the use of a fundamental and extremely useful new notion, that they call strength. After them, several other authors have been able to provide a proof utilizing various techniques; for instance, see \cite{ESS}. 
The known upper bounds either come from finiteness conditions related to Noetherianity and, as such, are not explicit (as in \cite{AH} and \cite{ESS}), or are expressed as huge towers of exponentials, with several layers. Even if the estimates are either not explicit or typically far from being optimal, the remarkable fact is that they exist in the first place.

The solution of Stillman's conjecture implies uniform upper bounds for several other invariants, including the Castelnuovo-Mumford regularity. Castelnuovo's theorem resembles Stillman's conjecture in what it provides a bound on a given invariant, the number of minimal quadratic generators of a non-degenerate prime $P \subseteq \kk[x_1,\ldots,x_n]$, which does not depend on $\kk$ (as long as it is algebraically closed), or on the number of variables $n$ of the ambient ring.

It is well-known that the bound of ${h+1 \choose 2}$ produced by Castelnuovo's theorem fails if $\kk$ is not algebraically closed (for instance, see \cite[Example 5.7]{DMV}, or Example \ref{Ex prime real numbers}). To the best of our knowledge, if no geometric assumptions such as $\kk$ being algebraically closed are involved, then there is no known upper bound for the number of quadratic minimal generators of a prime ideal just in terms of its height. One difficulty is that the usual approach of using Bertini's theorem fails. In fact, as pointed out before, linear forms can be introduced in the process of taking hyperplane sections and saturating the the ideal $P$, and this forces one to account for the number of possibly minimal quadratic generators killed by these linear forms. Such a number could very well depend on $n$. 

Even less seems to be known if one is not merely interested in quadrics, but wants to study the number of minimal homogeneous generators of $P$ of any degree $j \geq 2$. That is, if one wants to provide an upper bound on the $(0,j)$-th graded Betti number $\beta_{0,j}(P)$ just in terms of $j$ and of the height of $P$. Our main result answers precisely this question, with no assumptions on the base field. 

\begin{theoremx}[see Theorem  \ref{THM generators}] \label{THMX A}
Let $S$ be a standard graded polynomial ring over a field $\kk$, and $P \subseteq S$ be a homogeneous prime ideal of height $h$. For every $j \geq 0$ we have that 
\[
\beta_{0,j}(P) \leq h^{2^{j+1}-3}.
\]
\end{theoremx}


Our methods rely on a version developed in \cite{CV} of the 
well-known Buchberger's algorithm, which allows a direct computation of 
a Gr{\"o}bner basis with respect to any given weight-order without first refining it to a monomial one (see Algorithm \ref{BAlgorithm}), and Lemma \ref{Lemma reduced GB single elt}, which allows to extract a {\it minimal} Gr{\"o}bner basis from it. Our use of  this algorithm is more theoretical than computational. What is relevant to us is that it yields Proposition \ref{Prop upper bound}, which is the key for the inductive step in the proof of Theorem \ref{THMX A}. Another important observation involved in the proof is the fact that the minimal number of generators of a prime ideal in degree $j$ equals the number of minimal relations of degree $D+j$ of an almost complete intersection of degrees $d_1 \leq d_2 \leq \ldots \leq d_h \ll D$ (see Lemma \ref{Lemma Betti ACI}). It is important to point out that, while the validity of Stillman's conjecture allows to bound the number of minimal relations of an ideal in terms of the degrees of its generators, this observation cannot be applied to our scenario. In fact, no information on the degrees $d_1,\ldots,d_h$ of the almost complete intersection or on its multiplicity can be deduced when constructing it in Lemma \ref{Lemma Betti ACI}.

Theorem \ref{THMX A} exhibits an explicit upper bound of doubly-exponential type. While this estimate seems far from being optimal, especially for large values of $j$ or $h$, a doubly-exponential behavior is inevitable given our methods. Nevertheless, just like for the current known bounds for Stillman's conjecture, the main purpose of our result is to prove that a bound depending only on $j$ and $h$ actually exists.

In the case of quadrics, an {\it ad hoc} analysis of the quantities involved allows us to obtain a significantly more accurate upper bound than the one of Theorem \ref{THMX A}. This estimate, like the one of Castelnuovo's theorem, is quadratic in the height of the prime.

\begin{theoremx}[see Theorem \ref{THM quadrics}] \label{THMX B} Let $S$ be a standard graded polynomial ring over a field $\kk$, and $P \subseteq S$ be a non-degenerate homogeneous prime ideal of height $h$. The number of linearly independent quadrics contained in $P$ is at most $2h^2+h$.
\end{theoremx}

Finally, we extend Theorem \ref{THMX A} and produce explicit upper bounds for Betti numbers $\beta_{i,j}(I)$ of any radical ideal $I$ in terms of $i,j$, and the bigheight of $I$ (see Theorem \ref{THM Betti}). In order to achieve this, we use Theorem \ref{THMX A} and the strategies involved in its proof to show that, for {\it any} monomial order $\preccurlyeq$ and any unmixed radical ideal $I$ of height $h$, there is doubly-exponential upper bound for $\beta_{i,j}({\rm in}_{\preccurlyeq}(I))$ only depending on $i,j$ and the height of $I$ (see Proposition \ref{Prop Betti}).


\subsection*{Acknowledgments} We thank David Eisenbud, Mark Green and Hailong Dao for helpful discussions regarding the topics of this paper.

\subsection*{Notation and setup}
Throughout this article, $\kk$ is a field, and $S=\kk[x_1,\ldots,x_n] = \bigoplus_{j \geq 0}S_j$ is a graded polynomial ring over $\kk$, with $\deg(x_i)=1$ for every $i=1,\ldots,n$. We will refer to this as the standard grading on $S$. Given a homogeneous ideal $I \subseteq S$, and a non-negative integer $d$, we let $I_{\leq d}$ be the ideal generated by the homogeneous elements of $I$ of degree at most $d$. Given an integer $1 \leq h \leq n$, which will be either specified or clear from the context, we let $\ov{S}$ denote the subring $\kk[x_1,\ldots,x_h]$ of $S$. If $M = \bigoplus_j M_j$ is a finitely generated graded $S$-module, we let $\beta^S_{i,j}(M) = \dim_\kk(\Tor_i^S(M,\kk)_j)$ be the $(i,j)$-th graded Betti number of $M$ as an $S$-module, and we let $\beta_{i,\leq j}^S(M) = \sum_{t \leq j} \beta_{i,j}^S(M)$. We will drop the superscript when the ring over which we are computing Betti numbers is clear from the context.


\section{General coordinates and relative Buchberger's algorithm} \label{Section Buchberger}
Given $\boldsymbol{\omega} = (\omega_1,\ldots,\omega_n)  \in \NN^n$ and a monomial ${\bf X}^{\bf u} = x_1^{u_1} \cdots x_n^{u_n} \in S$, we let the weight of ${\bf X}^{\bf u}$ to be $\boldsymbol{\omega} \cdot {\bf u} = \sum_{i=1}^n \omega_i u_i$. This naturally induces a 
total preorder on monomials: if ${\bf X}^{\bf u}$ and ${\bf X}^{\bf v}$ are monomials in $S$, then ${\bf X}^{\bf u} \preccurlyeq_{\boldsymbol{\omega}} {\bf X}^{\bf v}$ if and only if $\boldsymbol{\omega} \cdot {\bf u} \leq \boldsymbol{\omega}\cdot {\bf v}$. If $f \in S$ is a polynomial, we can write it uniquely as a sum of monomials, with coefficients in $\kk$. We then define the initial form ${\rm in}_{\boldsymbol{\omega}}(f)$ of $f$ as the sum, with coefficients, of the monomials in the support of $f$ with maximal weight with respect to $\preccurlyeq_{\boldsymbol{\omega}}$. The initial ideal of $I$ with respect to $\preccurlyeq_{\boldsymbol{\omega}}$ is ${\rm in}_{\boldsymbol{\omega}}(I) = ({\rm in}_{\boldsymbol{\omega}}(g)\mid g \in I)$, the ideal generated by the initial forms of elements of $I$. If $g_1,\ldots,g_t$ is any system of generators of $I$, then clearly $({\rm in_{\boldsymbol{\omega}}}(g_i) \mid i=1,\ldots,t) \subseteq {\rm in}_{\boldsymbol{\omega}}(I)$, and it is well-known that the containment can be strict. We say that a collection of homogeneous generators $g_1,\ldots,g_t$ of a given ideal $I$ is a Gr{\"o}bner basis with respect to $\preccurlyeq_{\om}$ if equality holds above. Finally, we say that $g_1,\ldots,g_t$ is a {\it minimal} Gr{\"o}bner basis of $I$ if ${\rm in}_{\boldsymbol{\omega}}(g_1),\ldots,{\rm in}_{\boldsymbol{\omega}}(g_t)$ are a minimal set of generators of ${\rm in}_{\boldsymbol{\omega}}(I)$.

\subsection{Minimal Gr{\"o}bner bases for weight preorders}
We now focus on how to compute a minimal Gr{\"o}bner basis of a given homogeneous ideal $I \subseteq S$ with respect to a given weight. If $\preccurlyeq$ is a monomial order (hence ${\rm in}_{\preccurlyeq}(f)$ is a monomial for every $f \in S$), then a standard way to compute a Gr{\"o}bner basis of $I$ with respect to $\preccurlyeq$ is the well-known Buchberger's algorithm. This process involves the computation of the so-called S-pairs, which are obtained from syzygies between initial forms with respect to $\preccurlyeq$ of pairs of elements of $I$, and the calculation of remainders of certain divisions. 
In a finite number of steps, the Buchberger's algorithm produces a Gr{\"o}bner basis of $I$ with respect to $\preccurlyeq$.

If $\om \in \NN^n$ is a weight, the Buchberger's algorithm as stated could fail to produce a Gr{\"o}bner basis of $I$ with respect to $\preccurlyeq_{\om}$. In \cite[Section 4]{CV}, the first author and Varbaro produce a variant of Buchberger's algorithm which works for a weight order as well without refining it to a monomial order. 
In \cite{CV} the authors point out that they are more interested in the theoretical aspects of the algorithm they describe, rather than in the computational ones. The same is true for us. However, for our purposes, we need to revise \cite[Algorithm 4.2]{CV} and show how to obtain a {\it minimal} Gr{\"o}bner basis from it. We start by recalling some notation used in \cite{CV} to describe the algorithm.

Let $\widetilde{S} = S[y]$, and let $\boldsymbol{\omega} = (\omega_1,\ldots,\omega_n) \in \NN^n$ be a weight. We give bi-degrees $\deg(x_i) = (1,\omega_i)$ and $\deg(y) = (0,1)$ to the variables of $\widetilde{S}$. Given a non-zero polynomial $g = \sum_{\bf{u}} c_{\bf u} {\bf X^{\bf u}} \in S$ with $c_{\bf u} \in \kk$, let $d = \max\{\boldsymbol{\omega} \cdot {\bf u} \mid c_{\bf u} \ne 0\}$ be the largest weight of a monomial in its support. We let $\widetilde{g} = y^d \sum_{{\bf u}} c_{\bf u} y^{-\boldsymbol{\omega} \cdot {\bf u}} {\bf X}^{\bf u}$ be its homogenization in $\widetilde{S}$. Given $f \in \widetilde{S}$, we let $\ul{f} \in S$ be its evaluation at $y=0$, and given a non-zero element $f \in \widetilde{S}$, we let $\deg_{\ul{x}}(f)$ be its total degree in the variables $x_1,\ldots,x_n$. If $g \in S$ has degree $d$, then $\deg_{\ul{x}}(\widetilde{g}) = d$, and it coincides with the degree of $\ul{\widetilde{g}} \in S$. 
If $I \subseteq S$ is an ideal, we let $\widetilde{I} = (\widetilde{g} \mid g \in I)$ be the ideal of  $\widetilde{S}$ generated by the homogenization of all elements in $I$. On the other hand, given an ideal $J \subseteq \widetilde{S}$, we let $\ul{J}$ be the ideal of $S$ obtained as the image of $J$ under the evaluation map at $y=0$. It is well-known that, given $I \subseteq S$, if we evaluate $\widetilde{I}$ at $y=1$ we get back the ideal $I$. On the other hand, evaluating at $y=0$ one gets the initial ideal of $I$ with respect to $\boldsymbol{\omega}$, that is, $\ul{(\widetilde{I})} = {\rm in}_{\boldsymbol{\omega}}(I)$. For a reference of these facts, see for example \cite[Proposition 8.26]{MS}. Finally, if $I =(g_1,\ldots,g_t) \subseteq S$ is a homogeneous ideal, then 
$\widetilde{I} = (\widetilde{g_1},\ldots,\widetilde{g_t}):y^\infty$.

For the convenience of the reader, we briefly recall \cite[Algorithm 4.2]{CV}, which returns a Gr{\"o}bner basis with respect to any weight order $\om \in \NN^n$.

\begin{algorithm} \label{BAlgorithm}
Let $I \subseteq S$ be a homogeneous ideal, and $g_1,\ldots,g_t$ be a system of homogeneous generators of $I$. Let $J_0 = (\widetilde{g_1},\ldots,\widetilde{g_t})$. 
Given any free presentation of $S/\ul{J_0}$ over $S$, we lift it to a composition of maps
\[
\xymatrix{
\widetilde{S}^{s} \ar[rr]^-\Phi && \widetilde{S}^t \ar[rr]^-{[\widetilde{g_1}, \ldots, \widetilde{g_t}]} &&\widetilde{S}.
}
\]
Observe that the above might not even be a complex. The columns of $\Phi$ represent lifts to $\widetilde{S}$ of syzygies $\sigma_1,\ldots,\sigma_s$ of $\ul{\widetilde{g_1}},\ldots,\ul{\widetilde{g_t}}$. As shown in \cite[Algorithm 4.2]{CV}, we can write $[\widetilde{g_1}, \ldots, \widetilde{g_t}] \circ \Phi(\widetilde{S}^s)$ as the ideal $(y^{a_1} \psi_1,\ldots,y^{a_s} \psi_s)$ for some bi-homogeneous elements $\psi_i \in \widetilde{S} \smallsetminus (y)\widetilde{S}$, with $a_i>0$ for all $i=1,\ldots,s$. We will say that the element $\psi_i$ is obtained by {\it pushing forward} the syzygy $\sigma_i$. Moreover, if we let $Q=(\psi_1,\ldots,\psi_s)$, then either $y$ is a non-zero divisor on $\widetilde{S}/J_0$, in which case $g_1,\ldots,g_t$ was already a Gr{\"o}bner basis of $I$, or $Q \not\subseteq J_0$. In the latter case, set $J_1 = J_0+Q$. 
We will refer to the above process as {\it one iteration} of the algorithm.

Performing more iterations, we obtain an increasing chain of ideals $J_0 \subseteq J_1 \subseteq \ldots \subseteq \widetilde{S}$ which, since $\widetilde{S}$ is Noetherian, must eventually stabilize at $J_m$ for some $m$. As shown in \cite[Algorithm 4.2]{CV} we have that $J_m = \widetilde{I}$, and thus ${\rm in}_{\om}(I) = \ul{J_m}$. By construction, the set of generators of $J_m$ obtained in this process, when evaluated at $y=1$, produces a Gr{\"o}bner basis of $I$. Furthermore, when evaluated at $y=0$, it produces a set of generators of ${\rm in}_\omega(I)$, not necessarily minimal.
\end{algorithm}

Finally, to obtain a minimal Gr{\"o}bner basis from \cite[Algorithm 4.2]{CV}, we will need the following reduction lemma.

\begin{lemma} \label{Lemma reduced GB single elt}
Let $f_1,\ldots,f_t\in \widetilde{S} \smallsetminus (y)\widetilde{S}$ be bi-homogeneous elements, and let $I=(f_1,\ldots,f_t)$. Given a bi-homogeneous element $f \in \widetilde{S}$, there exists a bi-homogeneous element $g \in \widetilde{S}$ such that $(I,f):y^\infty = (I,g):y^\infty$, and either $g = 0$ or $\deg_{\ul{x}}(g) = \deg_{\ul{x}}(f)$ and $\ul{g} \notin \ul{I}$.
\end{lemma}
\begin{proof}
If $(I,f):y^\infty = I:y^\infty$ then we can set $g=0$. For the rest of the proof, assume that $(I,f):y^\infty \supsetneq I:y^\infty$. Among all bi-homogeneous elements $g \in \widetilde{S} \smallsetminus (y)\widetilde{S}$ such that $(I,f): y^\infty = (I,g): y^\infty$ and $\deg_{\ul{x}}(g) = \deg_{\ul{x}}(f)$, we choose one such that $\ul{g}$ has minimal weight. Note that our current assumptions guarantee that $g \ne 0$. We claim that $\ul{g} \notin \ul{I}$. In fact, if $\ul{g} \in \ul{I}$, we would be able to find homogeneous elements $s_1,\ldots,s_t \in S$ such that $\ul{g} + \sum_{i=1}^t s_i \ul{f_i} = 0$, $\ul{(\widetilde{s_i})} = s_i$, and either $s_i=0$ or $s_i\ul{f_i}$ has the same weight as $\ul{g}$. Lifting this relation to $\widetilde{S}$, we get that $g+\sum_{i=1}^t s_if_i \in (y)\widetilde{S}$, that is, there would exist a bi-homogeneous element $g' \in \widetilde{S} \smallsetminus (y)\widetilde{S}$ such that $g+\sum_{i=1}^t s_if_i = y^N g'$ for some $N \geq 1$. In particular, note that $\deg_{\ul{x}}(g') = \deg_{\ul{x}}(g) = \deg_{\ul{x}}(f)$, but $\ul{g'}$ has weight strictly smaller than $\ul{g}$. However, the above relation gives that $(I,g'):y^\infty = (I,g):y^\infty = (I,f):y^\infty$, which contradicts our minimal choice for the weight of $\ul{g}$. 
\end{proof}

\begin{definition}
Given a list $\Gamma= \{g_1,\ldots,g_t\}$ of homogeneous polynomials in $S$ of degrees $d_1,\ldots,d_t$, we let 
\[
\syz_S(\Gamma) = \ker \left( \bigoplus_{i=1}^t S(-d_i) \stackrel{[g_1, \ldots, g_t]}{\longrightarrow} S\right)
\]
be the $S$-module of the syzygies of $g_1,\ldots,g_t$. 
\end{definition}

We warn the reader that $\syz_S(\Gamma)$ does not necessarily coincide with the first module of syzygies $\syz_S(J)$ of the ideal $J$ generated by $g_1,\ldots,g_t$, since we are not assuming any minimality conditions on such polynomials. In fact, for every $d \in \ZZ$ we have that $\beta_{1,d}(J) = \beta_{0,d}(\syz_S(J)) \leq \beta_{0,d}(\syz_S(\Gamma))$. 

We now apply Algorithm \ref{BAlgorithm} and Lemma \ref{Lemma reduced GB single elt} to reach the main goal of this section.


\begin{proposition} \label{Prop upper bound} Let $\om \in \NN^n$, and $I \subseteq S$ be a homogeneous ideal. Let $g_1,\ldots,g_t$ be a set of homogeneous generators of $I$ of degrees $d_1 \leq \ldots \leq d_t$, such that $({\rm in}_{\om}(g_i) \mid i=1,\ldots,t) = {\rm in}_{\om}(I)_{\leq d_t}$. 
If we let $\Gamma = \{{\rm in}_{\om}(g_i) \mid i=1,\ldots,t\}$, then 
\[
\max\{\beta_{0,d_t+1}({\rm in}_{\om}(I)),\beta_{1,d_t+1}({\rm in}_{\om}(I))\} \leq \beta_{0,d_t+1}(\syz_S(\Gamma)).
\] 
\end{proposition}
\begin{proof}
Performing one iteration of Algorithm \ref{BAlgorithm} to $g_1,\ldots,g_t$, with the same notation used therein we obtain bi-homogeneous elements $\psi_1,\ldots,\psi_s \in \widetilde{S} \smallsetminus (y)\widetilde{S}$. By construction, such elements are the push forward of homogeneous syzygies $\sigma_1,\ldots,\sigma_s$ of the elements $\ul{\widetilde{g_1}},\ldots,\ul{\widetilde{g_t}}$. Since $\kk$-linearly dependent syzygies would give rise to $\kk$-linearly dependent push forwards, without loss of generality we can assume that $\sigma_1,\ldots,\sigma_s$ are $\kk$-linearly independent. Let $\delta_i = \deg_{\ul{x}}(\psi_i)$. By possibly relabeling such elements, we may assume that 
\[
\delta_1 \leq \ldots \leq \delta_r \leq d_t < \delta_{r+1} = \ldots = \delta_{r+u} = d_t+1<\delta_{r+u+1} \leq \ldots \leq \delta_s.
\]
We let $J = (\widetilde{g_1},\ldots,\widetilde{g_t})$, and we apply Lemma \ref{Lemma reduced GB single elt} to $J$ together with each element $\psi_i$, for $1 \leq i \leq r$. We then find elements $\gamma_1,\ldots,\gamma_r \in \widetilde{S}$ such that $(J,\psi_i):y^\infty = (J,\gamma_i):y^\infty$. Moreover, either $\gamma_i = 0$, or $\deg_{\ul{x}}(\gamma_i) = \deg_{\ul{x}}(\psi_i) \leq d_t$ and $\ul{\gamma_i} \notin \ul{J}$. Since by assumption ${\rm in}_{\om}(g_1),\ldots,{\rm in}_{\om}(g_t)$ already generate $\ul{J}= {\rm in}_{\om}(I)_{\leq d_t}$, we must have $\gamma_i=0$ for all $1 \leq i \leq r$. Note that the condition that $(J,\psi_i):y^\infty = (J,\gamma_i):y^\infty = J:y^\infty$ implies that the elements $\psi_1,\ldots,\psi_r$ can be disregarded in a subsequent iterations of Algorithm \ref{BAlgorithm}. 

We now apply Lemma \ref{Lemma reduced GB single elt} to $J$ and $\psi_{r+1}$ to obtain a bi-homogeneous element $\psi'_{1}$ such that $(J,\psi_{r+1}):y^\infty = (J,\psi'_{1}):y^\infty$; moreover, $\psi_1'$ is either zero or $\deg_{\ul{x}}(\psi_{1}') = d_t+1$ and $\ul{\psi_{1}'} \notin \ul{J}$. By successively applying Lemma \ref{Lemma reduced GB single elt} to $(J,\psi_1',\ldots,\psi_i')$ and $\psi_{r+i+1}$ for every $1 \leq i \leq u$ we find bi-homogeneous elements $\psi_{1}',\ldots,\psi_{u}' \in \widetilde{S}$ such that $(J,\psi_{r+1},\ldots\psi_{r+u}):y^\infty = (J,\psi_1',\ldots,\psi_u'):y^\infty$; in addition, such elements are either zero or they  have $\ul{x}$-degree equal to $d_t+1$, and they satisfy $\ul{\psi_{i+1}'} \notin \ul{(J,\psi_{1}',\ldots,\psi_i')}$. By only picking the non-zero elements among them,  
we finally obtain bi-homogeneous elements $\varphi_{1},\ldots,\varphi_v \in \widetilde{S} \smallsetminus (y)\widetilde{S}$, with $v \leq u$, of $\ul{x}$-degree $d_t+1$, such that 
the images of the elements $\ul{\varphi_{1}},\ldots,\ul{\varphi_{v}}$ inside $S/{\rm in}_\omega(I)_{\leq d_t}$ are minimal generators of the ideal they generate in such a ring. Furthermore, Lemma \ref{Lemma reduced GB single elt} guarantees that 
$(J,\varphi_1,\ldots,\varphi_v):y^\infty = (J,\psi_{r+1},\ldots,\psi_{r+u}):y^\infty$.

We can now repeat Algorithm \ref{BAlgorithm} with the elements $g_1,\ldots,g_t,\varphi_1,\ldots,\varphi_v,\psi_{r+u+1},\ldots,\psi_s$ as input. This returns the same elements $\psi_1,\ldots,\psi_s$ obtained before, together with new elements $\theta_1,\ldots,\theta_r \in \widetilde{S}\smallsetminus (y)\widetilde{S}$ obtained by the Algorithm by pushing forward syzygies that involve at least one of the elements $\ul{\varphi_1},\ldots,\ul{\varphi_v},\ul{\psi_{r+u+1}},\ldots,\ul{\psi_s}$. Since the images of $\ul{\varphi_1},\ldots,\ul{\varphi_v}$ inside $S/{\rm in}_\omega(I)_{\leq d_t}$ are $\kk$-linearly independent, we must have that $\deg_{\ul{x}}(\theta_i)>d_t+1$ for every $i$. Using Lemma \ref{Lemma reduced GB single elt} as before, we see once again that $\psi_1,\ldots,\psi_r$ can be disregarded in a subsequent iteration of Algorithm \ref{BAlgorithm}. Moreover, it is now clear that $(J,\varphi_1,\ldots,\varphi_v):y^\infty = (J,\varphi_1,\ldots,\varphi_v,\psi_{r+i}):y^\infty$ for every $1 \leq i \leq u$, and thus also the elements $\psi_{r+1},\ldots,\psi_{r+u}$ can be disregarded in the next iteration. We now see that any further iteration of Algorithm \ref{BAlgorithm}, together with the considerations we just made, does not return any new element in $\ul{x}$-degree at most $d_t+1$. However, as the algorithm must eventually returns a Gr\"obner basis of ${\rm in}_\omega(I)$, we conclude that ${\rm in}_\omega(g_1),\ldots,{\rm in}_\omega(g_t),\ul{\varphi_1},\ldots,\ul{\varphi_v}$ must generate ${\rm in}_{\omega}(I)$ in degree up to $d_t+1$.  
Thus, we have that $\beta_{0,d_t+1}(\syz_S(\Gamma)) \geq u \geq v = \beta_{0,d_t+1}({\rm in}_{\om}(I))$. 

Now, if $\sigma$ is any minimal generator of $\syz_S({\rm in}_{\om}(I))$ of degree $d_t+1$, then for degree reasons $\sigma$ must be a syzygy between minimal generators of ${\rm in}_{\om}(I)$ of degree at most $d_t$. Because of our assumptions, we therefore have that $\sigma \in \syz_S(\Gamma)$. If $\sigma$ was not a minimal generator of $\syz_S(\Gamma)$, a fortiori it would not be a minimal generator of $\syz_S({\rm in}_{\om}(I)_{\leq d_t})$. This shows that $\beta_{1,d_t+1}({\rm in}_{\om}(I)) = \beta_{0,d_t+1}(\syz_S({\rm in}_{\om}(I)_{\leq d_t})) \leq \beta_{0,d_t+1}(\syz_S(\Gamma))$.
\end{proof}

\subsection{General revlex preorders and complete intersections}

We recall how to define total preorders on monomials starting from a matrix. Let $\Omega$ be an $m \times n$ matrix with non-negative integer entries, and let $\boldsymbol{\omega}_i$ denote its $i$-th row. Then $\Omega$ induces a total preorder on monomials: we declare that ${\bf X}^{\bf u} \preccurlyeq_{\Omega} {\bf X}^{\bf v}$ if and only if either $\boldsymbol{\omega}_i \cdot {\bf u} = \boldsymbol{\omega}_i \cdot {\bf v}$  for all $i=1,\ldots,m$, or there is $1 \leq j < m$ such that $\boldsymbol{\omega}_i \cdot {\bf u} = \boldsymbol{\omega}_i \cdot {\bf v}$ for all $1 \leq i \leq j$, and $\boldsymbol{\omega}_{j+1} \cdot {\bf u} < \boldsymbol{\omega}_{j+1} \cdot {\bf v}$. As a consequence, we can talk about initial forms, and the initial ideal with respect to $\preccurlyeq_\Omega$, which we denote by ${\rm in}_\Omega(-)$. 

Given a matrix preorder $\preccurlyeq_\Omega$ and a finite set $\mathcal M$ of monomials, one can always find a weight $\boldsymbol{\omega}$ (depending on the set $\mathcal M$) such that for any $m_1, m_2 \in \mathcal M$ one has $m_1 \preccurlyeq_\Omega m_2$ if and only if $m_1 \preccurlyeq_{\boldsymbol{\omega}} m_2$. Thus, when computing ${\rm in}_\Omega(I)$ of a given homogeneous ideal $I$, by Noetherianity one can always reduce to computing ${\rm in}_{\boldsymbol{\omega}}(I)$ for some weight $\boldsymbol{\omega}$. In particular, one can use Algorithm \ref{BAlgorithm} and Lemma \ref{Lemma reduced GB single elt} in order to produce a minimal Gr{\"o}bner basis of a given ideal $I$ with respect to $\preccurlyeq_\Omega$.

We now introduce a special matrix preorder, which is extensively used in this article. Given integers $1 \leq h < n$, consider the $(n-h) \times n$ matrix
\[
\Omega_h = \left[\begin{array}{ccc|cccccc}
1 & \ldots &1 & 1 & 1 & \ldots  & 1 & 1 & 0 \\
1 & \ldots & 1 & 1  & 1&   \ldots  & 1& 0 & 0 \\
\vdots & \vdots & \vdots &\vdots & \vdots & \vdots & \vdots & \vdots \\
1 & \ldots & 1 & 1 & 0 & \ldots & 0 & 0 & 0 \\
1 & \ldots & 1 & 0 & 0 &\ldots & 0 & 0 & 0
\end{array}\right].
\]

We denote the matrix order induced by $\Omega_h$ simply by $\preccurlyeq_h$. 
Observe that $\preccurlyeq_1$ 
is nothing but the standard degree-revlex order on monomials of $S$. Given a homogeneous ideal $I$, we will denote by ${\rm in}_h(I)$ the initial ideal of $I$ with respect to $\preccurlyeq_h$, that is, the ideal generated by the initial forms ${\rm in}_h(f)$ of all polynomials $f \in I$.

\begin{lemma} \label{Lemma Initial CI}
Let $f_1,\ldots,f_h \in S=\kk[x_1,\ldots,x_n]$ be homogeneous polynomials, and assume that $f_1,\ldots,f_h,x_{h+1},\ldots,x_n$ is a full regular sequence in $S$. Then ${\rm in}_h(f_1,\ldots,f_h,x_{h+1},\ldots,x_n) = ({\rm in}_h(f_1),\ldots,{\rm in}_h(f_h),x_{h+1},\ldots,x_n)$. In particular, $f_1,\ldots,f_h$ are a Gr{\"o}bner basis with respect to 
$\preccurlyeq_h$, and ${\rm in}_h(f_1),\ldots,{\rm in}_h(f_h)$ are still a homogeneous regular sequence in $\ov{S} = \kk[x_1,\ldots,x_h]$. 
\end{lemma}
\begin{proof}
It suffices to prove the first statement, which is just a consequence of well-known commutation properties of revlex-type preorders with modding out the last variables (for instance, see \cite[15.7]{Eisenbud}).
\end{proof}

\section{A non-standard grading and the $\Lambda$-construction} \label{Section Lambda}
We introduce a grading on $S=\kk[x_1,\ldots,x_n]$ which lies in between the standard grading and the monomial $\NN^n$-grading.  We view $S$ as $\ov{S}[x_{h+1},\ldots,x_n]$, and 
for $i=0,\ldots,n-h$ we let $\eta_i \in \NN^{n-h+1}$ be the vector with $1$ in position $i+1$ and $0$ elsewhere. We let an element of $\ov{S}_j$ have degree $j \cdot \eta_0$ in $S$, while for $h+1 \leq i \leq n$ we let $x_i$ have degree $\eta_{i-h}$. Extending this accordingly defines the desired $\NN \times \NN^{n-h}$-grading on $S$ 
for any $1 \leq h \leq n$. 
Observe that the choice $h=n$ gives the standard grading on $S$, while $h=1$ corresponds to the monomial $\NN^n$-grading. 

Given an element $g \in S$ which is homogeneous with respect to this new grading, we denote by $\deg_{(1,1)}(g) \in \NN^{n-h+1}$ its $\NN \times \NN^{n-h}$-degree. Observe that we can write $\deg_{(1,1)}(g) = (e,\epsilon)$ for some $e \in \NN$ and $\epsilon \in \NN^{n-h}$. We then let $\deg_{(1,0)}(g) = e$ and $\deg_{(0,1)}(g) = \epsilon$. We still denote by $\deg(g) = e+|\epsilon|$, where $|\epsilon|$ denotes the sum of the entries of $\epsilon \in \NN^{n-h}$, the degree of $g$ with respect to the standard grading in $S$, and we sometimes refer to it as the total degree of $g$ to distinguish it from the previous ones.

\begin{remarks} It follows directly from the given definitions that:
\begin{enumerate}[(i)]
\item A polynomial of $S$ is homogeneous with respect to the $\NN \times \NN^{n-h}$-grading if and only if it is the product of a homogeneous element in $\ov{S}$ with a monomial in the last $n-h$ variables.
\item If $\Gamma=\{g_1,\ldots,g_t\}$ is a list of are $\NN \times \NN^{n-h}$-graded polynomials in $S$, say $\deg_{(1,1)}(g_i) = \eta_i$, then $\syz_S(\Gamma)$ is an $\NN \times \NN^{n-h}$-graded submodule of 
$\bigoplus_{i=1}^t S(-\eta_i)$.
\item If $I$ is an ideal in $S$ which is homogeneous with respect to the standard grading, then ${\rm in}_h(I)$ is $\NN \times \NN^{n-h}$-graded.
\end{enumerate}
\end{remarks}

Given a list of $\NN \times \NN^{n-h}$-graded polynomials $\Gamma= \{\gamma_1,\ldots,\gamma_t\}$. For $1 \leq r \leq t$ we write $\gamma_r = \ov{\gamma_r}m_r$ for some homogeneous element $\ov{\gamma}_r \in \ov{S}$ and a monic monomial $m_r \in \kk[x_{h+1},\ldots,x_n]$. Let $L_r = \{{\rm lcm}(m_{i_1},\ldots,m_{i_s},m_r) \mid \{i_1,\ldots,i_s\} \subseteq \{1,\ldots,r-1\}\}$. 
We consider the following free $\ov{S}$-module
\[
\Lambda(\Gamma) = \bigoplus_{r=1}^t \bigoplus_{m \in L_r} \ov{S} \cdot e_m^r.
\]
We give to an element $e_m^r \in L_r$ of the basis degree equal to $\deg(\gamma_r) + \deg(m) - \deg(m_r) = \deg(m) + \deg(\ov{\gamma_r})$. Observe that the $j$-th graded component of $\Lambda(\Gamma)$ is
\[
\Lambda(\Gamma)_j = \bigoplus_{r=1}^t \left(\bigoplus_{m \in L_r} \ov{S}_{j-\deg(m)-\deg(\ov{\gamma_r})} \cdot e_m^r \right).
\]

Set $\lambda_j(\Gamma) = \dim_\kk(\Lambda(\Gamma)_j)$. We start with a dimension count, which inductively allows to control $\lambda_j(\Gamma)$ as one appends elements to the list $\Gamma$.

\begin{lemma} \label{Lemma lambda count} Let $\Gamma_1$ be a list of $\NN \times \NN^{n-h}$-graded polynomials of $S$ of degree at most $j-1$, and $\Gamma_2$ be a list of $\NN \times \NN^{n-h}$-graded polynomials of degree $j$ such that $|\Gamma_2| \leq \lambda_j(\Gamma_1)$. Let $\Gamma$ be the list obtained by appending $\Gamma_2$ to $\Gamma_1$. Then
\[
\lambda_{j+1}(\Gamma) \leq 2 \lambda_j(\Gamma_1)^2 +(2h-1)\lambda_j(\Gamma_1).
\]
\end{lemma}
\begin{proof}
Let $\Gamma_1=\{\gamma_1,\ldots,\gamma_t\}$ and $\Gamma_2=\{\gamma_{t+1},\ldots,\gamma_{t+s}\}$. Observe that, by construction of $\Lambda(\Gamma_1)$, we have that $|\Gamma_1| = t \leq \lambda_j(\Gamma_1)$. Write $\gamma_i = \ov{\gamma_i}m_i$ for some homogeneous $\ov{\gamma_i} \in \ov{S} = \kk[x_1,\ldots,x_h]$ and monic monomials $m_i \in \kk[x_{h+1},\ldots,x_n]$. We can write
\[
\Lambda(\Gamma)_{j+1} = \left(\bigoplus_{r=1}^t\bigoplus_{m \in L_r} \ov{S}_{j+1-\deg(m)-\deg(\ov{\gamma_r})} \cdot e_m^r \right) \oplus \left(\bigoplus_{r={1}}^s\bigoplus_{m \in L_{t+r}} \ov{S}_{j+1-\deg(m)-\deg(\ov{\gamma_{t+r}})} \cdot e_m^r \right) =: V \oplus W.
\]
Observe that 
\[
V = \left(\bigoplus_{r=1}^t\bigoplus_{m \in L_r, \deg(e_m^r) \leq j} \ov{S}_{j+1-\deg(m)-\deg(\ov{\gamma_r})} \cdot e_m^r\right) \oplus \left(\bigoplus_{r=1}^t\bigoplus_{m \in L_r, \deg(e_m^r)=j+1} \kk \cdot e_m^r\right).
\]
The dimension of the first summand of $V$ is
\begin{align*}
\sum_{r=1}^t \sum_{m \in L_r, \deg(e_m^r) \leq j} \dim_\kk \ov{S}_{j+1-\deg(m)-\deg(\ov{\gamma_r})} \\
\leq \sum_{r=1}^t \sum_{m \in L_r, \deg(e_m^r) \leq j} h \dim_\kk \ov{S}_{j-\deg(m)-\deg(\ov{\gamma_r})} = h \lambda_j(\Gamma_1).
\end{align*}
The dimension of the second summand of $V$ is bounded above by all the possible least common multiples between an element $m' \in L_r$ for some $1 \leq r \leq t$ such that $\deg(e_{m'}^r) \leq j$, together with a monomial $m_i \in \{m_1,\ldots,m_r\}$. This number is bounded above by the possible choices of unordered pairs of distinct elements $\{m,m'\} \subseteq \{\widetilde{m} = {\rm lcm}(m_{i_1},\ldots,m_{i_s}) \mid \{i_1,\ldots,i_s\} \subseteq \{1,\ldots,t\}, \deg(e_{\widetilde{m}}^s) \leq j\}$. 
It is immediate to see from the description of $\Lambda(\Gamma_1)$ that the cardinality of the latter is at most  
$\lambda_j(\Gamma_1)$, and therefore the dimension of the second summand of $V$ is at most ${\lambda_j(\Gamma_1) \choose 2}$.


The dimension of $W$ is bounded above by three contributions. The first comes from the summand
\[\bigoplus_{r=1}^s\left(\bigoplus_{m \in L_{t+r}, \deg(e_m^{t+r})=j} \ov{S}_1 e_m^{t+r}\right) = \bigoplus_{r=1}^s \ov{S}_1 e_{m_{t+r}}^{t+r},
\] 
whose dimension is $\sum_{r=1}^s \dim_\kk(\ov{S}_1) = h |\Gamma_2| \leq h \lambda_j(\Gamma_1)$. 

Since $\deg(e_{m_{t+r}}^{t+r}) = j$ for every $1 \leq r \leq s$, the second and the third contributions to $\dim_\kk(W)$ come from least common multiples $m$ performed between a monomial $m_{t+r}$ for some $1 \leq r \leq s$ and another monomial which either belongs to $\{m_1,\ldots,m_t\}$ or to $\{m_{t+1},\ldots,m_{t+r-1}\}$, in such a way that $\deg(e_m^{t+r}) = j+1$. Note that there are at most $|\Gamma_2| \cdot |\Gamma_1| \leq \lambda_j(\Gamma_1)^2$ many possible least common multiples coming from the first case scenario, while there are at most
${|\Gamma_2| \choose 2} \leq {\lambda_j(\Gamma_1) \choose 2}$ possible ones coming from the second. Putting all estimates together,  we conclude that
\[
\lambda_{j+1}(\Gamma) \leq 2 \lambda_j(\Gamma_1)^2 + (2h-1) \lambda_j(\Gamma_1). \qedhere
\]
\end{proof}

The following proposition justifies the introduction of $\Lambda(\Gamma)$.

\begin{proposition} \label{Prop bound lambda}
Let $\varphi_1,\ldots,\varphi_h \in \ov{S}$ be homogeneous elements, and let $D > \max\{j \in \NN \mid \beta_{1,j}(\varphi_1,\ldots,\varphi_h)\ne 0\}$. Let $\gamma_1,\ldots,\gamma_t$ be $\NN \times \NN^{n-h}$ homogeneous elements of degrees at least $D$ and at most $D+j$ for some $j \geq 0$. Set $\Gamma = \{\gamma_{1},\ldots,\gamma_t\}$ and let $J=(\varphi_1,\ldots,\varphi_h,\gamma_1,\ldots,\gamma_t)$. Then $\beta_{1,D+j+1}(J) \leq \lambda_{D+j+1}(\Gamma)$.
\end{proposition}
\begin{proof}
Write $\gamma_i = \ov{\gamma_i}m_i$ for some monic monomials $m_i \in \kk[x_{h+1},\ldots,x_n]$, and homogeneous elements $\ov{\gamma_i} \in \ov{S}$. Given an $\NN \times \NN^{n-h}$-graded syzygy $\sigma$ of $\varphi_1,\ldots,\varphi_h,\gamma_1,\ldots,\gamma_t$, we can represent it as a $(h+t)$-uple 
\[
\sigma=(s_1,\ldots,s_{h+t}) \in \left(\bigoplus_{i=1}^h S\left(-(\deg(\varphi_i),0)\right)\right)\oplus \left(\bigoplus_{i=1}^t S\left(-(\deg(\ov{\gamma_i}),\epsilon_i)\right)\right),
\]
where $\epsilon_i$ represents the exponent vector of the monomial $m_i$ with respect to the variables $x_{h+1},\ldots,x_n$. 
If $\sigma \ne 0$, we let $W(\sigma) = \max\{i \mid s_i \ne 0\}$. Given a finite set $\Sigma$ of $\NN\times \NN^{n-h}$-graded syzygies of $\Gamma$, we let $W(\Sigma) = \sum_{\sigma \in \Sigma} W(\sigma)$. 
Among all sets $\Sigma$ of $\NN\times \NN^{n-h}$-graded elements which minimally generate $\syz_S(J)$ up to degree $D+j+1$, we pick one which minimizes $W(\Sigma)$. 

We claim that there is an injective map of $\kk$-vector spaces $\psi: (\kk \Sigma)_{D+j+1} \hookrightarrow \Lambda(\Gamma)_{D+j+1}$, defined on a $\kk$-basis as follows. First of all, we note that any element $\sigma \in \Sigma$ must satisfy $W(\sigma) > h$, otherwise it would correspond to a minimal syzygy of $\varphi_1,\ldots,\varphi_h$ of degree $D+j+1$, contradicting our choice of $D$. For every $1 \leq r \leq t$, we let $\Sigma_r=\{\sigma \in \Sigma \mid W(\sigma)=h+r\}$. For $\sigma = (s_1,\ldots,s_h,\sigma_1,\ldots,\sigma_t) \in \Sigma_r$ we can write $\sigma_{r} = \ov{\sigma_{r}} m_{r}'$ for some monic monomial $m_r' \in \kk[x_{h+1},\ldots,x_n]$ and some element $\ov{\sigma_r} \in \ov{S}_{D+j+1-\deg(\gamma_r) - \deg(m_r')}$. Since $\sigma \in \Sigma_r$ is part of a minimal graded generating set of $\syz_S(J)$, the monomial $m_rm_r'$ must be obtained as the least common multiple of  some of the monomials $m_1,\ldots,m_{r-1}$, together with $m_r$. In other words, $m:=m_rm_r' \in L_r$, and we can set $\psi(\sigma) = (\ov{\sigma_r}m)e_m^r \in \Lambda(\Gamma)_{D+j+1}$. To show that this map is injective, let $\sigma^{(1)},\ldots,\sigma^{(s)}$ be a $\kk$-basis of $(\kk \Sigma)_{D+j+1}$, and assume by way of contradiction that $\psi(\sigma^{(1)}),\ldots,\psi(\sigma^{(s)})$ are $\kk$-linearly dependent. Without loss of generality, we may assume that we have non-zero elements $\alpha_1,\ldots,\alpha_s \in \kk$ and a $\kk$-linear combination $\alpha_1 \psi(\sigma^{(1)}) + \ldots + \alpha_s\psi(\sigma^{(s)}) = 0$, where all the elements $\sigma^{(i)}$ belong to $\Sigma_r$ for some $1 \leq r \leq t$, and $\sigma^{(i)}_r = \ov{\sigma_r}^{(i)}m_r'$ for all $i$. We let $m=m_rm_r'$. We have
\[
0 = \sum_{i=1}^s \alpha_i \psi(\sigma^{(i)}) = \left(\sum_{i=1}^s \alpha_i \ov{\sigma_r}^{(i)}\right)e_m^r.
\]
It follows that $\sigma' = \sum_{i=1}^s \alpha_1 \sigma^{(i)}$ is an $\NN \times \NN^{n-h}$ which, together with $\sigma^{(2)},\ldots,\sigma^{(s)}$, constitutes a minimal $\NN \times \NN^{n-h}$-graded generating set $\Sigma'$ of $\syz_S(J)$ in degree up to $D+j+1$ such that $W(\Sigma') < W(\Sigma)$. This contradicts our choice of $\Sigma$. Therefore $\psi$ is injective, and the proof is complete.
\end{proof}

\section{Bounds on Betti numbers of prime ideals}  \label{Section bounds}

We start this main section of the article with an observation that, even if rather elementary, will come very handy in the study of generators of prime ideals. We say that an ideal of $S$ is unmixed if all its associated primes are minimal, and they all have the same height. In the next lemma we show that, to compute the graded Betti numbers of an unmixed homogeneous radical ideal of height $h>0$, one can always reduce to studying the Betti numbers of a suitable almost complete intersection, that is, an ideal of height $h$ generated by $h+1$ elements. 
\begin{lemma} \label{Lemma Betti ACI}
Assume that $\kk$ is infinite. Let $I \subsetneq S = \kk[x_1,\ldots,x_n]$ be a homogeneous unmixed radical ideal of height $h<n$, and $\f \subseteq I$ be an ideal generated by a regular sequence with $\Ht(\f) =h$. 
If $I \ne \f$, there exists an integer $D_0$ with the following property: for any $D>D_0$, there exists a homogeneous element $g \in S_D$ such that $\beta_{0,j}(I) = \beta_{1,D+j}(\f+(g))$ for all $j \geq 0$.
\end{lemma}
\begin{proof}
Let $D' = \max\{j \in \ZZ \mid \beta_{0,j}(\f) \ne 0 \text{ or } \beta_{1,j}(\f) \ne 0\}$. Observe that the minimal primes of $I$ are contained in the minimal primes of $\f$. Since $I$ is radical, unmixed and homogeneous, we can find a homogeneous element $g \in S$ such that $I = \f:g$. 
Let $D_0= \max\{D',\deg(g)\}$, and fix an integer $D>D_0$. Since ${\rm depth}(S/I)>0$ and $\kk$ is infinite, we can find a linear form $\ell$ which is regular modulo $I$. In particular, $I=\f:g\ell^N$ for any $N \geq 0$. After replacing $g$ with $g\ell^{D-\deg(g)}$, we may assume without loss of generality that $D= \deg(g)$. Let $\a = \f+(g)$. After applying the functor $-\otimes_S \kk$ to the graded exact sequence $0 \to S/I(-D) \to S/\f \to S/\a \to 0$, and looking at the component of degree $m$ in the long exact sequence of $\Tor_\bullet^S(-,\kk)$ modules, we get an exact sequence of $\kk$-vector spaces
\[
\Tor_2^S(S/\f,\kk)_m \to \Tor_2^S(S/\a,\kk)_m \to \Tor_1^S(S/I,\kk)_{m-D} \to \Tor_1^S(S/\f,\kk)_m.
\]
For $j\geq 0$ and $m =j+D > D_0$ we have that $\Tor_2 ^S(S/\f,\kk)_m = \Tor_1^S(S/\f,\kk)_m = 0$. It follows that $\beta_{0,j}(I) = \dim_{\kk}(\Tor_1^S(S/I,\kk))_j = \dim_{\kk}(\Tor_2^S(S/\a,\kk))_{j+D} = \beta_{1,j+D}(\a)$, as claimed.
\end{proof}


We are finally ready to prove our main theorem.

\begin{theorem} \label{THM generators} Let $I \subsetneq S$ be a homogeneous unmixed radical ideal of height $h$. For all $j \geq 0$ we have that
\[
\beta_{0,j}(I) \leq h^{2^{j+1}-3}.
\]
\end{theorem}
\begin{proof}
The cases $h \leq 1$ or $j=0$ are trivially satisfied, therefore we will assume that $h \geq 2$ and $j \geq 1$. The case in which $h=n = \dim(S)$ is also trivial, since in this case $I$ is forced to be equal to $(x_1,\ldots,x_n)$ and the claimed bound is satisfied. Assume that $h<n$. A suitable extension of the base field does not affect our assumptions and the desired conclusion, therefore we may assume that $\kk$ is infinite. We may find an ideal $\f \subseteq I$, with $\Ht(\f) = h$, generated by a homogeneous regular sequence of degrees $d_1 \leq \ldots \leq d_h$. If $I = \f$, then $\beta_{0,j}(I) \leq h$ for every $j \geq 1$, and the claimed bound is satisfied. Assume that $I \ne \f$. By Lemma \ref{Lemma Betti ACI}, for $D \gg 0$ we can find a homogeneous element $g \notin \f$ of degree $D$ such that $\beta_{0,j}(I)  = \beta_{1,D+j}(\a)$, where $\a=\f+(g)$ is an almost complete intersection of height $h$. In particular, as seen in the proof of Lemma \ref{Lemma Betti ACI}, we assume that $D > \max\{j \in \ZZ \mid \beta_{1,j}(\f) \ne 0\}$.

By upper semi-continuity it suffices to prove that $\beta_{1,D+j}({\rm in}_h(\a)) \leq h^{2^{j+1}-3}$. After a sufficiently general change of coordinates, if we let $\varphi_1={\rm in}_h(f_1),\ldots,\varphi_h={\rm in}_h(f_h) \in \ov{S}$, then by Lemma \ref{Lemma Initial CI} we may assume that  
$\varphi_1,\ldots,\varphi_h,x_{h+1},\ldots,x_n$ form a regular sequence. 


We prove by induction on $j \geq 1$ that there exists a list $\Gamma$ of $\NN \times \NN^{n-h}$-graded elements that generate ${\rm in}_h(\a)$ in degrees between $D$ and $D+j-1$, and such that $\lambda_{D+j}(\Gamma) \leq h^{2^{j+1}-3}$. It will then follow from Propositions \ref{Prop upper bound} and \ref{Prop bound lambda} that $\beta_{1,D+j}({\rm in}_h(\a)) \leq h^{2^{j+1}-3}$, and this will conclude the proof.


The base case $j=1$ is satisfied, since we can choose $\Gamma = \{{\rm in}_h(g)\}$, and it follows that $\lambda_{D+1}(\Gamma) = h = h^{2^{j+1}-3}$. 

Let $j \geq 2$, and assume that the claimed inequality is true for $j-1$. Let $\Gamma_1$ be a set of minimal $\NN \times \NN^{n-h}$-graded generators of ${\rm in}_h(\a)$ in degrees between $D$ and $D+j-2$, and $\Gamma_2$ be minimal $\NN \times \NN^{n-h}$-graded generators of ${\rm in}_h(\a)$ of degree precisely equal to $D+j-1$. Let $\Gamma$ be the list obtained by appending $\Gamma_2$ to $\Gamma_1$. It follows from Proposition \ref{Prop bound lambda} that $|\Gamma_2| = \beta_{0,D+j-1}({\rm in}_h(\a)) \leq \lambda_{D+j-1}(\Gamma_1)$. 
By Lemma \ref{Lemma lambda count} we have that $\lambda_{D+j}(\Gamma) \leq 2\lambda_{D+j-1}(\Gamma_1)^2 + (2h-1)\lambda_{D+j-1}(\Gamma_1)$. Since $\lambda_{D+j-1}(\Gamma_1) \leq h^{2^j-3}$ by induction, and because $2^j-1 \leq 2^{j+1}-5$ for $j \geq 2$, we conclude that 
\begin{align*}
\lambda_{D+j}(\Gamma) & \leq 2\left(h^{2^j-3}\right)^2 + (2h-1)h^{2^j-3} \leq   2h^{2^{j+1}-6} + h^2h^{2^j-3} \\
& \leq h^{2^{j+1}-5} + h^{2^j-1} \leq 2h^{2^{j+1}-5} \leq h^{2^{j+1}-3}. \qedhere
\end{align*}
\end{proof}

In the case of quadrics, Theorem \ref{THM generators} provides $h^{5}$ as an upper bound, which is definitely larger compared to the one of ${h+1 \choose 2}$ given by Castelnuovo's Theorem. The following example shows that one cannot expect the bound of Castelnuovo to hold for unmixed radical ideals. 

\begin{example} \label{Ex radical h2} For any $h \geq 1$, the ideal $(x_1,\ldots,x_h) \cap (y_1,\ldots,y_h) \subseteq S=\kk[x_1,\ldots,x_h,y_1,\ldots,y_h]$ is unmixed and radical, and it is minimally generated by the $h^2$ monomials $\{x_iy_j \mid 1 \leq i,j \leq h\}$.
\end{example}
 
If one does not assume that $\kk$ is algebraically closed, the upper bound of Castelnuovo's theorem does not even hold for prime ideals.

\begin{example} \label{Ex prime real numbers} Let $S=\RR[a,b,c,d]$ and $P=(a^2+c^2,b^2+d^2,ad-bc,ab+cd)$. Observe that $S/P \cong \RR[x,y,ix,iy]$, therefore $P$ is a prime ideal of height two. However, $P$ is minimally generated by $4>{3 \choose 2}$ quadrics.
\end{example}

The next theorem provides a more refined upper bound for the number of quadratic minimal generators of {\it any} unmixed radical ideal in a standard graded polynomial ring over {\it any} field. In order to achieve this, we run a more careful analysis of the first few inductive steps in the proof of Theorem \ref{THM generators}.

\begin{theorem} \label{THM quadrics}
Let $I \subsetneq S$ be an unmixed radical ideal of height $h$. Then $\beta_{0,2}(I) \leq 2h^2 + h$.
\end{theorem}
\begin{proof}
As in the proof of Theorem \ref{THM generators}, we may assume that $h \geq 2$ and that $\kk$ is infinite. 

Let $\a=(f_1,\ldots,f_h,g)$ be an almost complete intersection of degrees $d_1 \leq \ldots \leq d_h \ll D$ such that $\beta_{0,j}(I) = \beta_{1,D+j}(\a)$ for all $j$, constructed as in Lemma \ref{Lemma Betti ACI}. After performing a sufficiently general change of coordinates, our goal is to prove that $\beta_{1,D+2}({\rm in}_h(\a)) \leq 2h^2+h$. 

Let $g_1=g, g_2,\ldots,g_t$ be homogeneous elements of $S$ such that, if $\varphi_i = {\rm in}_h(f_i)$ and $\gamma_i = {\rm in}_h(g_i)$, then the elements $\varphi_1,\ldots,\varphi_h,\gamma_1,\ldots,\gamma_t$ minimally generate ${\rm in}_h(\a)$ in degree up to $D+1$. We let $\Gamma_1= \{\gamma_1\}$, $\Gamma_2=\{\gamma_2,\ldots,\gamma_t\}$, and $\Gamma = \{\gamma_1,\gamma_2,\ldots,\gamma_t\}$. Observe that by Proposition \ref{Prop bound lambda} we have that $t-1 = |\Gamma_2| \leq \lambda_{D+1}(\Gamma_1) = h$.

 Write $\gamma_i = \ov{\gamma_i} m_i$ for some monomials $m_i \in \kk[x_{h+1},\ldots,x_n]$ and homogeneous elements $\ov{\gamma}_i \in \ov{S}$. 


We have seen in the proof of Lemma \ref{Lemma lambda count} that $\lambda_{D+2}(\Gamma)$ is bounded above by the contributions of two vector spaces, $V$ and $W$. In the current notation we have that $V=\ov{S}_2 m_1$, and thus $\dim_\kk(V) = {h+1 \choose 2}$. On the other hand, we have already seen in the proof of Lemma \ref{Lemma lambda count} that $\dim_\kk(W) \leq  h \lambda_{D+1}(\Gamma_1) + {\lambda_{D+1}(\Gamma_1) \choose 2} + |\Gamma_2| \cdot |\Gamma_1| \leq h^2 + {h \choose 2} + h$. We conclude by Proposition \ref{Prop bound lambda} that $\beta_{1,D+2}({\rm in}_h(\a)) \leq \lambda_{D+2}(\Gamma) \leq 2h^2+h$.
\end{proof}
We conclude the article by providing an upper bound on the Betti numbers $\beta_{i,j}(I)$ of any radical ideal $I$ which only depends only on $i,j$ and its bigheight, i.e., on the largest height among the minimal primes of $I$. In particular, if $P$ is a prime ideal, $\beta_{i,j}(P)$ can be bounded above by a function depending only on $i,j$ and the height of $P$. 

We first show that, if we know an upper bound on the number of generators of an ideal $I$ in every degree, then we can accordingly bound the graded Betti numbers $\beta_{i,j}({\rm in}_{\preccurlyeq}(I))$ of its initial ideal with respect to {\it any} monomial order $\preccurlyeq$. This fact is embedded in the proof of the next result.


\begin{proposition} \label{Prop Betti} Let $I \subsetneq S$ be a homogeneous ideal and let $a,h \geq 1$ be integers such that $\beta_{0,j}(I) \leq h^{2^{j+a}-3}$ for all $j \geq 0$. 
For all $i,j \geq 0$ we have that
\[
\displaystyle \beta_{i, \leq i+j}(I) \leq {h^{2^{a+j+1}-3} \choose i+1}.
\]
\end{proposition} 
\begin{proof}
The case $h \leq 1$ is trivially satisfied, therefore we may assume that $h \geq 2$. Let ${\rm in}(-)$ denote the initial ideal with respect to any monomial order, and let $J={\rm in}(I)$. By upper semi-continuity, it suffices to prove the claimed upper bound for $\beta_{i,\leq i+j}(J)$. We first prove the case $i=0$ by induction on $j \geq 0$. The case $j=0$ is trivial. Assume the claim is proved for some $j \geq 0$, and let us show it for $j+1$. Let $\Gamma$ be a set of homogeneous elements of $I$ whose initial forms minimally generate ${\rm in}(I)$ up to degree $j$. By inductive hypothesis, we have that $|\Gamma| = \beta_{0,\leq j}(J) \leq h^{2^{a+j+1}-3}$. Since we are dealing with a monomial order, every S-pair in Buchberger's algorithm involves only two polynomials at a time. As every minimal generator of degree $j$ of $J$ is either the initial form of the reduction of one such S-pair between elements of $\Gamma$, or it is the initial form of a minimal generator of $I$ of degree $j+1$, we conclude by induction that 
\[
\beta_{0,j+1}(J) \leq {h^{2^{a+j+1}-3} \choose 2} + h^{2^{a+j+1}-3} \leq h^{2^{a+j+3}-6} + h^{2^{a+j+1}-3}.
\]
It follows that $\beta_{0,\leq j+1}(J) = \beta_{0,\leq j}(J) + \beta_{0,j+1}(J) \leq h^{2^{a+j+1}-3} + h^{2^{a+j+2}-6} + h^{2^{a+j+1}-3} \leq h^{2^{a+j+2}-3}$.

For the claim about Betti numbers, observe first that, by degree considerations, we have that $\beta_{i,\leq i+j}(J) = \beta_{i,\leq i+j}(J_{\leq j})$. Since the Taylor complex $T_\bullet$ on a generating set of $J_{\leq j}$ is a (possibly non minimal) free resolution $S/J_{\leq j}$, it follows that
\[
\displaystyle \beta_{i,\leq i+j}(J) \leq \rank(T_{i+1}) = {\beta_0(J_{\leq j}) \choose i+1} \leq {h^{2^{a+j+1}-3} \choose i+1}. \qedhere
\]
\end{proof}

Ideals of bigheight one are principal, so this case is trivial. The next result gives an upper bound on the graded Betti numbers of {\it any} radical ideal of bigheight $h\geq 2$.

\begin{theorem} \label{THM Betti}
Let $I \subsetneq S$ be a radical ideal of bigheight $h \geq 2$. For all $i,j \geq 0$ we have that 
\[
\displaystyle \beta_{i, \leq i+j}(I) \leq {h^{2^{j+h}-2} \choose i+2}.
\]
\end{theorem}
\begin{proof}
For any $1 \leq \ell \leq h$, we let $I_\ell$ be the intersection of all minimal primes of $I$ of height $\ell$, so that $I = \bigcap_{\ell=1}^h I_\ell$. If we let $f$ be a generator of $I_1$, then since $I$ is radical we must have $I=f I'$, with $I'= \bigcap_{\ell=2}^h I_\ell$. It follows that $\beta_{i,j}(I) = \beta_{i,j-\deg(f)}(I')$. Since our upper bound is an increasing function in the variable $j$, it suffices to show the upper bound for the graded Betti numbers of $I'$. Thus, without loss of generality, we assume that $f=1$, so that $I=\bigcap_{\ell=2}^h I_\ell$.

We proceed by induction on $h \geq 2$. If $h=2$, then Theorem \ref{THM generators} yields $\beta_{0,j}(I) \leq h^{2^{j+1}-3}$ and Proposition \ref{Prop Betti} gives that
\[
\beta_{i,\leq i+j}(I) \leq {h^{2^{j+2}-3} \choose i+1} \leq {h^{2^{j+h}-2}-1 \choose i+1} \cdot \frac{h^{2^{j+h}-2}}{i+2} = {h^{2^{j+h}-2} \choose i+2}.
\]

Assume $h \geq 3$, and let $J = \bigcap_{\ell=2}^{h-1} I_\ell$ and $L=J+I_h$. Consider the graded short exact sequence
\[
\xymatrix{
0 \ar[r] & S/I \ar[r] & S/J \oplus S/I_h \ar[r] & S/L \ar[r] & 0.
}
\]
Applying the functor $- \otimes_S \kk$ and counting dimensions of the graded components of the long exact sequence of $\Tor_\bullet^S(-,\kk)$ we obtain that $\beta_{i, \leq i+j}(I) \leq \beta_{i,\leq i+j}(J) + \beta_{i,\leq i+j}(I_h) + \beta_{i+1, \leq i+j}(L)$.

By induction, we have that $\beta_{i, \leq i+j}(J) \leq {(h-1)^{2^{j+h-1}-2} \choose i+2}$, and in particular we have that $\beta_{0, j}(J) \leq {(h-1)^{2^{j+h-1}-2} \choose 2} \leq h^{2^{j+h}-4}$. By Theorem \ref{THM generators} we have that $\beta_{0, j}(I_h) \leq h^{2^{j+1}-3}$, and therefore $\beta_{0,  j}(L) \leq \beta_{0,j}(J) + \beta_{0,j}(I_h) \leq h^{2^{j+h}-3}$ for every $j$. By Proposition \ref{Prop Betti} we get $\beta_{i,\leq i+j}(I_h) \leq {h^{2^{j+2}-3} \choose i+1}$ and $\beta_{i+1,\leq i+j}(L) \leq {h^{2^{j+h}-3} \choose i+2}$. Putting all these estimates together, some easy calculations show that
\begin{align*}
\beta_{i, \leq i+j}(I) & \leq {(h-1)^{2^{j+h-1}-2} \choose i+2}  + {h^{2^{j+2}-3} \choose i+1} + {h^{2^{j+h}-3} \choose i+2} \\
& \leq {h^{2^{j+h}-3}-1 \choose i+2}  +{h^{2^{j+h}-3} + 1 \choose i+2}  \leq {h^{2^{j+h}-2} \choose i+2}. \qedhere
\end{align*}

\end{proof}

\bibliographystyle{alpha}
\bibliography{References}
\end{document}